\documentclass{article}

\usepackage{bold-extra}
\usepackage{mathptmx,amsmath,amscd,amssymb,amsthm,xspace}
\usepackage[all,tips]{xy}
\usepackage[dvips]{graphicx}
\usepackage{graphics,epsfig,wrapfig,verbatim,syntonly}
\usepackage{hyperref,color, url,fancyhdr,tikz-cd}
\usepackage{adjustbox}
\usepackage{soul}
\usepackage{textcomp}
\usepackage{tikz}
\usetikzlibrary{arrows.meta}
\usepackage{circuitikz}
\usepackage{menukeys}

\usepackage{comment}
\usepackage{quiver}
\usepackage{mathtools}
\usepackage{placeins}
\usepackage{filehook}
\usepackage[T1]{fontenc}
\usepackage{listings}
\usepackage{etoc}
\usepackage{makeidx}

\theoremstyle{plain}

\newtheorem{theorem}{Theorem}[section]
\newtheorem{definition}[theorem]{Definition}
\newtheorem{proposition}[theorem]{Proposition}

\ifcsname lemma\endcsname\else\newtheorem{lemma}[theorem]{Lemma}\fi
\ifcsname corollary\endcsname\else\newtheorem{corollary}[theorem]{Corollary}\fi
\newtheorem{maintheorem}[theorem]{Theorem}
\newtheorem{maincorollary}[theorem]{Corollary}
\newtheorem*{unnumberedquestion}{Question}
\theoremstyle{remark}
\ifcsname remark\endcsname\else\newtheorem{remark}[theorem]{Remark}\fi
\ifcsname example\endcsname\else\newtheorem{example}[theorem]{Example}\fi
\theoremstyle{plain}

\newcommand{\disc}{\operatorname{disc}}

\newcommand{\Ram}{\operatorname{Ram}}

\newcommand{\M}{\operatorname{M}}

\providecommand{\Ocal}{\mathcal{O}}
\providecommand{\afrak}{\mathfrak{a}}
\providecommand{\pfrak}{\mathfrak{p}}
\providecommand{\tr}{\operatorname{tr}}
\providecommand{\trd}{\operatorname{trd}}
\providecommand{\nrd}{\operatorname{nrd}}
\providecommand{\SL}{\operatorname{SL}}
\providecommand{\GL}{\operatorname{GL}}
\providecommand{\PSL}{\operatorname{PSL}}
\providecommand{\PGL}{\operatorname{PGL}}

\providecommand{\vol}{\operatorname{vol}}

\providecommand{\tbt}[4]{\begin{pmatrix} #1 & #2 \\ #3 & #4 \end{pmatrix}}
\newcommand{\slfrak}{\mathfrak{sl}}
\newcommand{\Cl}{\operatorname{Cl}}
\newcommand{\CL}{\operatorname{CL}}
\newcommand{\LS}{\operatorname{L}}
\newcommand{\kinf}{K_\infty}

\makeindex

\begin{document}
\setcounter{tocdepth}{3}
\clearpage{}


\title{Isospectral congruence manifolds, spectral rigidity, \\ and a converse to Vign\'eras}
\author{Justin Katz, D.~B.~McReynolds}
\maketitle


\begin{abstract}
\noindent Let $B$ be a quaternion division algebra over a number field $k$ that gives rise to arithmetic Fuchsian or Kleinian groups, and let $\Ocal$ be a maximal order of $B$. We show that every closed orientable Riemannian manifold isospectral to a principal congruence hyperbolic surface $X_\Ocal(\afrak)$ is isometric to $X_{\Ocal'}(\afrak)$ for some maximal order $\Ocal'$ of $B$, provided that no prime factor of $\afrak$ ramifies in $B$ and that $\afrak$ satisfies a mild condition at the primes over $2$. This is a converse to the construction of Vign\'eras, which produces isospectral manifolds from maximal orders of different types. In particular the spectral genus of $X_\Ocal(\afrak)$ has at most $t(B)$ elements, and $X_\Ocal(\afrak)$ is spectrally rigid when the type number $t(B)$ is one. In dimension three we prove the same statements for manifolds that, in addition, have the same complex lengths of closed geodesics. The proof rests on a local result that characterizes, by traces alone, the subgroups of $\SL(2)$ over a $p$--adic ring that are conjugate into a principal congruence subgroup by traces. As applications, every principal congruence Hurwitz surface is spectrally rigid, among them Klein's quartic, the Fricke--Macbeath surface and the surfaces of the first Hurwitz triplet, which partially answers a question of Reid. 
\end{abstract}

\section{Introduction}

Two closed Riemannian manifolds are \textbf{isospectral} if their Laplace--Beltrami operators have the same eigenvalue spectrum with multiplicities. The \textbf{spectral genus} of a closed orientable Riemannian manifold $X$ is the set of isometry classes of closed orientable Riemannian manifolds isospectral to $X$, and $X$ is \textbf{spectrally rigid} if its spectral genus is a single point. Since Milnor's isospectral flat tori \cite{Milnor-[EigenvaluesLaplaceOperator]1964}, general constructions of isospectral non-isometric manifolds have been given by Vign\'eras \cite{Vigneras-[VarietesRiemanniennesIsospectrales]1980}, Sunada \cite{Sunada-[RiemannianCoveringsIsospectral]1985} and Gordon \cite{Gordon-[IsospectralClosedRiemannian]1993}. On the other hand, spectral genera are generally small. They are finite for hyperbolic surfaces \cite{McKean-[SelbergsTraceFormula]1972}, compact for metrics on a surface \cite{Osgood.Phillips.Sarnak-[CompactIsospectralSets]1988}, and contain finitely many topological types under curvature bounds \cite{Perry.Petersen.Brooks-[CompactnessFinitenessTheorems]1992}. By Wolpert \cite{Wolpert-[LengthSpectraModuli]1979}, a generic hyperbolic surface is spectrally rigid, yet few explicit spectrally rigid surfaces are known. Natural candidates are the \textbf{Hurwitz surfaces}, the closed Riemann surfaces of genus $g\geq 2$ with $84(g-1)$ conformal automorphisms, and Reid asked the following \cite[Question 7.4]{Reid-[TracesLengthsAxes]2014}.

\begin{unnumberedquestion}
Is a Hurwitz surface determined by its spectrum?
\end{unnumberedquestion}

The first isospectral non-isometric hyperbolic surfaces and $3$--manifolds are due to Vign\'eras. They are arithmetic and come from a quaternion algebra $B$ whose maximal orders, although all locally conjugate, fall into $t(B)>1$ global conjugacy classes; this is called the \textbf{type number} of $B$. In the absence of selectivity \cite{Chinburg.Friedman-[EmbeddingTheoremQuaternion]1999}, maximal orders of different types have the same trace data, and the Selberg trace formula turns them into isospectral, generally non-isometric, manifolds. Our main result is a converse to Vign\'eras. 


Let $k$ be a number field with ring of integers $R$, and let $B$ be a quaternion division algebra over $k$. We say that $B$ is of \textbf{Fuchsian type} if $k$ is totally real and $B$ is unramified at exactly one real place, and of \textbf{Kleinian type} if $k$ has exactly one complex place and $B$ is ramified at every real place. Put $\kinf = \mathbf{R}$ or $\mathbf{C}$ accordingly, and let $\psi\colon B\to\M(2,\kinf)$ be the embedding given by the unramified archimedean place. For a maximal order $\Ocal$ and a nonzero ideal $\afrak$ of $R$, let $\Gamma_\Ocal(\afrak)$ be the image in $\PSL(2,\kinf)$ of
\begin{align*}
\Ocal^1(\afrak) = \{x \in \Ocal : \nrd(x) = 1,\ x \equiv 1 \bmod \afrak\Ocal\},
\end{align*}
and put $X_\Ocal(\afrak) = \Gamma_\Ocal(\afrak)\backslash\mathbf{H}^n$. We assume:
\begin{enumerate}
\item[(H1)] The algebra $B$ is of Fuchsian or Kleinian type and $\Ocal$ is a maximal order of $B$.
\item[(H2)] No prime factor of $\afrak$ ramifies in $B$, and every prime $\pfrak\mid\afrak$ over $2$ satisfies $e(\pfrak\mid 2) = 1$ or $v_\pfrak(\afrak)\leq e(\pfrak\mid 2)$.
\item[(H3)] The group $\Gamma_\Ocal(\afrak)$ is torsion free.
\end{enumerate}
If $2$ is unramified in $k$, then (H2) only asks that $\afrak$ be prime to $\disc(B)$. Hypothesis (H3) holds, for instance, whenever $\afrak$ is divisible by a prime over a rational prime unramified in $k$ (Lemma \ref{lemma:torsion-free}).

\begin{maintheorem}\label{thm:A}
Assume \textup{(H1)--(H3)}, with $B$ of Fuchsian type. Every closed orientable Riemannian manifold isospectral to $X_\Ocal(\afrak)$ is isometric to $X_{\Ocal'}(\afrak)$ for some maximal order $\Ocal'$ of $B$. In particular the spectral genus of $X_\Ocal(\afrak)$ has at most $t(B)$ elements.
\end{maintheorem}

In dimension three the Laplace spectrum determines the lengths of closed geodesics but not their complex lengths. Write $\CL(M)$ for the set of complex lengths of closed geodesics of a hyperbolic $3$--manifold $M$.

\begin{maintheorem}\label{thm:B}
Assume \textup{(H1)--(H3)}, with $B$ of Kleinian type, and let $M$ be a closed orientable Riemannian manifold isospectral to $X = X_\Ocal(\afrak)$. Then $M$ is hyperbolic, and if $\CL(M)\subseteq\CL(X)$, then $M$ is isometric to $X_{\Ocal'}(\afrak)$ for some maximal order $\Ocal'$ of $B$. In particular this holds for every closed orientable hyperbolic $3$--manifold $M$ with the same complex length spectrum as $X$.
\end{maintheorem}

The last assertion holds because the complex length spectrum determines the Laplace spectrum \cite{Salvai-[LaplaceComplexLength]2002,Kelmer-[RefinementStrongMultiplicity]2014}. When $t(B) = 1$, all the manifolds $X_{\Ocal'}(\afrak)$ are isometric to $X_\Ocal(\afrak)$, and we recover spectral rigidity.

\begin{maincorollary}\label{cor:C}
Assume \textup{(H1)--(H3)} and $t(B) = 1$. If $B$ is of Fuchsian type, then $X_\Ocal(\afrak)$ is spectrally rigid. If $B$ is of Kleinian type, then $X_\Ocal(\afrak)$ is determined up to isometry, among closed orientable hyperbolic $3$--manifolds, by its complex length spectrum.
\end{maincorollary}

Corollary \ref{cor:C} applies to the following examples (\S\ref{sec:applications}).

\begin{maincorollary}\label{cor:hurwitz}
Every principal congruence cover $X(\afrak)$ of the $(2,3,7)$ triangle orbifold, where $\afrak \neq \mathbf{Z}[\eta]$ is a nonzero ideal of $\mathbf{Z}[\eta]$ and $\eta = 2\cos(2\pi/7)$, is a spectrally rigid Hurwitz surface. In particular Klein's quartic, the Fricke--Macbeath surface, and the surfaces of the first Hurwitz triplet are spectrally rigid.
\end{maincorollary}

\begin{maincorollary}\label{cor:shimura}
Let $B$ be an indefinite quaternion division algebra over $\mathbf{Q}$ of discriminant $D$, and let $\Ocal$ be a maximal order of $B$. For every integer $N>1$ prime to $D$, the Shimura curve $X_\Ocal(N)$ is spectrally rigid.
\end{maincorollary}

\subsubsection*{Acknowledgements and AI Disclosure:} This paper is based on the PhD thesis of the first author. We would like to thank Michelle Chu, Ben Linowitz, Alan Reid, and Matthew Stover. The second author used Opus 5.5 during the post-production stage to clean up the draft and for references. The authors are responsible for all of the content of this paper.

\section{Quaternion algebras and orders}\label{sec:quaternion}

We collect the facts about quaternion algebras that we need and refer the reader to \cite{Maclachlan.Reid-[ArithmeticHyperbolic3Manifolds]2003} and \cite{Voight-[QuaternionAlgebras]2021}.

\subsection{Quaternion algebras}

A \textbf{quaternion algebra} over a field $k$ of characteristic zero is a central simple $k$--algebra $B$ of dimension $4$. It carries a canonical involution $x\mapsto x^*$, and the \textbf{reduced trace} and \textbf{reduced norm} are
\begin{align}\label{eq:reduced-trace-norm}
\trd(x) = x + x^*, \qquad \nrd(x) = x x^*, \qquad x\in B.
\end{align}
Every $x\in B$ satisfies $x^2 - \trd(x) x + \nrd(x) = 0$. The reduced trace is $k$--linear, the reduced norm is multiplicative and is a quadratic form on $B$ whose associated bilinear form is $(x,y)\mapsto \trd(xy^*)$, and the bilinear form $(x,y)\mapsto \trd(xy)$ is nondegenerate. For every $x \in B$ one has
\begin{align}\label{eq:norm-of-sum}
\nrd(1+x) = 1 + \trd(x) + \nrd(x).
\end{align}
The algebra $B$ is a division algebra if and only if $\nrd$ is anisotropic. Otherwise $B\cong \M(2,k)$, and then $\trd$ and $\nrd$ become the trace and the determinant. We write $B^1$ for the group of elements of reduced norm $1$.

\subsection{Local fields}

Let $F$ be a local field of characteristic zero. If $F=\mathbf{C}$, all quaternion algebras over $F$ are isomorphic to $\M(2,\mathbf{C})$. If $F$ is $\mathbf{R}$ or a finite extension of $\mathbf{Q}_p$, there is, up to isomorphism, a unique quaternion division algebra $D_F$ over $F$ \cite[\S 2.6]{Maclachlan.Reid-[ArithmeticHyperbolic3Manifolds]2003}. For $F = \mathbf{R}$ it is Hamilton's quaternions $\mathcal{H}$, and $\mathcal{H}^1 \cong \mathrm{SU}(2)$ is compact. A quaternion algebra $B$ over $F$ is \textbf{ramified} if $B\cong D_F$ and \textbf{unramified} if $B\cong \M(2,F)$. Let $F$ be nonarchimedean with ring of integers $R$. An $R$--\textbf{order} of a quaternion algebra $B$ over $F$ is a subring $\Ocal$ containing $1$ that is a finitely generated $R$--module and spans $B$ over $F$. Every $x \in \Ocal$ is integral over $R$, and so has reduced trace and reduced norm in $R$. An order is \textbf{maximal} if it is not properly contained in another order.

\begin{proposition}[{\cite[\S 6.4, \S 6.5]{Maclachlan.Reid-[ArithmeticHyperbolic3Manifolds]2003}}]\label{prop:local-orders}
Let $F$ be a nonarchimedean local field with ring of integers $R$.
\begin{enumerate}
\item If $B$ is ramified, the set of elements of $B$ that are integral over $R$ is the unique maximal order of $B$.
\item Every maximal order of $\M(2,F)$ is conjugate under $\GL(2,F)$ to $\M(2,R)$. 
\end{enumerate}
In particular any two maximal orders of $B$ are conjugate under $B^\times$.
\end{proposition}

\subsection{Number fields}

Let $k$ be a number field with ring of integers $R$. For a place $\nu$ of $k$ we write $k_\nu$ for the completion, and for a finite prime $\pfrak$ we write $R_\pfrak$ for the completion of $R$ at $\pfrak$. If $B$ is a quaternion algebra over $k$ we put $B_\nu = B\otimes_k k_\nu$ and write $\Ram(B)$ for the set of places $\nu$ at which $B_\nu$ is ramified, $\Ram_\infty(B)$ and $\Ram_f(B)$ for its archimedean and finite parts. The set $\Ram(B)$ is finite of even cardinality, contains no complex place, and determines $B$ up to isomorphism. Conversely every finite set of noncomplex places of even cardinality arises in this way \cite[\S 7.3]{Maclachlan.Reid-[ArithmeticHyperbolic3Manifolds]2003}. The algebra $B$ is a division algebra if and only if $\Ram(B)\neq \emptyset$, and the \textbf{discriminant} $\disc(B)$ is the product of the primes in $\Ram_f(B)$. An \textbf{order} of $B$ is a subring containing $1$ that is a finitely generated $R$--module spanning $B$ over $k$, and it is \textbf{maximal} if it is maximal under inclusion. Every order is contained in a maximal order. If $\Ocal$ is an order then $\trd(\Ocal) \subseteq R$ and $\nrd(\Ocal)\subseteq R$, and $\Ocal^* = \Ocal$ since $x^* = \trd(x) - x$. We shall use the local--global principle for lattices repeatedly. An $R$--lattice in $B$ is a finitely generated $R$--submodule $L$ that spans $B$. For a finite prime $\pfrak$ we write $L_\pfrak = L\otimes_R R_\pfrak \subseteq B_\pfrak$.

\begin{proposition}[{\cite[\S 9.4]{Voight-[QuaternionAlgebras]2021}}]\label{prop:local-global-lattices}
Let $L$ be an $R$--lattice in $B$.
\begin{enumerate}
\item $L = \{x \in B : x \in L_\pfrak \text{ for all primes } \pfrak\}$.
\item Let $S$ be a finite set of primes and, for each $\pfrak\in S$, let $M(\pfrak)\subseteq B_\pfrak$ be an $R_\pfrak$--lattice. Then there is a unique $R$--lattice $L'$ in $B$ with $L'_\pfrak = M(\pfrak)$ for $\pfrak\in S$ and $L'_\pfrak = L_\pfrak$ for $\pfrak\notin S$.
\end{enumerate}
Moreover $L$ is a (maximal) order if and only if $L_\pfrak$ is a (maximal) order for every $\pfrak$.
\end{proposition}

For an ideal $\afrak = \prod \pfrak^{n_\pfrak}$ of $R$ and an order $\Ocal$, the lattice $\afrak\Ocal$ satisfies $(\afrak\Ocal)_\pfrak = \pfrak^{n_\pfrak}\Ocal_\pfrak$. Hence, by Proposition \ref{prop:local-global-lattices},
\begin{align}\label{eq:aO-local}
\afrak\Ocal = \{x \in \Ocal : x \in \pfrak^{n_\pfrak}\Ocal_\pfrak \text{ for every } \pfrak \mid \afrak\}.
\end{align}

\subsection{Type numbers}\label{sec:type-number}

The group $B^\times$ acts on the set of maximal orders of $B$ by conjugation. The number $t(B)$ of orbits is finite and is called the \textbf{type number} of $B$. By Proposition \ref{prop:local-orders}, all maximal orders are locally conjugate, so $t(B)$ measures the failure of the local--global principle for conjugacy of maximal orders, much as the class number measures the failure of the local--global principle for principality of ideals. We say that $B$ satisfies the \textbf{Eichler condition} if it is unramified at some archimedean place. Quaternion algebras of Fuchsian or Kleinian type satisfy it. Let $P_B$ be the group of principal fractional ideals $xR$ with $x\in k^\times$ positive at every real place in $\Ram_\infty(B)$, and let $\Cl_B = I_k / P_B$, where $I_k$ is the group of fractional ideals. This is a quotient of the narrow class group of $k$.

\begin{proposition}[{\cite[Theorem 6.7.6]{Maclachlan.Reid-[ArithmeticHyperbolic3Manifolds]2003}}]\label{prop:type-number}
If $B$ satisfies the Eichler condition, then $t(B)$ is the order of the quotient of $\Cl_B$ by the subgroup generated by the squares and by the classes of the primes in $\Ram_f(B)$.
\end{proposition}

\begin{corollary}\label{cor:typenumberone}
Let $B$ satisfy the Eichler condition. If $\Cl_B$ has odd order, in particular if the narrow class number of $k$ is odd, then $t(B) = 1$. This holds if $k$ has class number one and its units realize every sign pattern at the real places in $\Ram_\infty(B)$.
\end{corollary}

\begin{proof}
A finite abelian group of odd order is generated by its squares, which gives the first two assertions. If $h_k = 1$ and every sign pattern at $\Ram_\infty(B)$ is realized by a unit, every principal ideal has a generator positive at $\Ram_\infty(B)$, so $\Cl_B$ is trivial.
\end{proof}

\subsection{Congruence subgroups}

Let $\Ocal$ be an order of $B$ and $\afrak$ a nonzero ideal of $R$. We put $\Ocal^1 = \Ocal\cap B^1$ and $\Ocal^1(\afrak) = \Ocal^1 \cap (1 + \afrak\Ocal)$, the \textbf{principal congruence subgroup} of level $\afrak$. It is the kernel of the reduction map $\Ocal^1 \to (\Ocal/\afrak\Ocal)^\times$, so it is normal in $\Ocal^1$. It is also normalized by every $\beta\in B^\times$ that normalizes $\Ocal$.

\begin{lemma}\label{lemma:congruence-traces}
If $\gamma\in \Ocal^1(\afrak)$, then $\trd(\gamma) - 2 \in \afrak^2$.
\end{lemma}

\begin{proof}
Write $\gamma = 1+x$ with $x = \sum_i a_i x_i$, where $a_i\in \afrak$ and $x_i\in\Ocal$. By \eqref{eq:norm-of-sum}, $1 = \nrd(\gamma) = 1 + \trd(x) + \nrd(x)$, so $\trd(\gamma) - 2 = \trd(x) = -\nrd(x)$. Expanding the quadratic form,
\begin{align*}
\nrd(x) = \sum_i a_i^2 \nrd(x_i) + \sum_{i<j} a_ia_j \trd(x_i x_j^*),
\end{align*}
and $\nrd(x_i)$ and $\trd(x_ix_j^*)$ lie in $R$ because $x_i x_j^* \in \Ocal$. Hence $\nrd(x)\in \afrak^2$.
\end{proof}

\begin{lemma}\label{lemma:torsion-free}
Put $c(\ell) = \ell-1$ for an odd prime $\ell$ and $c(2) = 2$. Suppose that $\afrak$ is divisible by a prime $\pfrak$ of residue characteristic $\ell$ with $e(\pfrak \mid \ell) < c(\ell)$. If $\gamma\in\Ocal^1(\afrak)$ and $\gamma^m = \pm 1$ for some $m\geq 1$, then $\gamma = \pm 1$. In particular the image of $\Ocal^1(\afrak)$ in $B^\times/k^\times$ is torsion free.
\end{lemma}

\begin{proof}
Suppose that $\gamma^m = \pm 1$ and $\gamma\neq\pm 1$. Then $\gamma$ has finite order. Over a splitting field, $\gamma$ is semisimple with eigenvalues $\zeta,\zeta^{-1}$ for a root of unity $\zeta$, and $\zeta\neq \pm 1$ because $\gamma \neq \pm 1$ and $\gamma$ is semisimple. Let $r\geq 3$ be the order of $\zeta$, let $L = k(\zeta)$ and let $\mathfrak{P}$ be a prime of $L$ over $\pfrak$. Then $t = \trd(\gamma) = \zeta+\zeta^{-1}$ and
\begin{align*}
2 - t = -\zeta^{-1}(1-\zeta)^2 .
\end{align*}
By Lemma \ref{lemma:congruence-traces}, $v_{\mathfrak{P}}(2-t) \geq 2 e(\mathfrak{P}\mid\pfrak)$. If $r$ is not a prime power, or is a power of a prime different from $\ell$, then $1-\zeta$ is a unit at $\mathfrak{P}$, which is absurd. Thus $r = \ell^s$, in which case $v_{\mathfrak{P}}(1-\zeta) = e(\mathfrak{P}\mid \ell)/\varphi(\ell^s)$. We obtain $2e(\mathfrak{P}\mid\ell)/\varphi(\ell^s) \geq 2e(\mathfrak{P}\mid \pfrak)$, that is $e(\pfrak\mid\ell) \geq \varphi(\ell^s)$. Since $r\geq 3$, we have $\varphi(\ell^s)\geq c(\ell)$, contrary to the hypothesis.
\end{proof}

Finally we record a criterion for a group of units to generate an order.

\begin{lemma}\label{lemma:order-generated}
Let $\Delta \leq B^1$ be a subgroup that spans $B$ over $k$ and satisfies $\trd(\Delta)\subseteq R$. Then the $R$--span $R[\Delta]$ of $\Delta$ is an order of $B$. In particular $\Delta \leq \Ocal'^1$ for some maximal order $\Ocal'$.
\end{lemma}

\begin{proof}
The $R$--span is a subring containing $1$ that spans $B$, so we must show it is finitely generated. Choose $\gamma_1,\dots,\gamma_4\in\Delta$ forming a $k$--basis of $B$, and let $G = (\trd(\gamma_i\gamma_j))_{i,j}$. Then $d=\det G \in R$ and $d\neq 0$ since the trace form is nondegenerate. If $x = \sum_i c_i\gamma_i\in R[\Delta]$ with $c_i\in k$, then $x\gamma_j$ is an $R$--linear combination of elements of $\Delta$, so $\trd(x\gamma_j) \in R$ for all $j$. Solving $G (c_i)_i = (\trd(x\gamma_j))_j$ shows $c_i \in d^{-1}R$. Hence $R[\Delta]\subseteq d^{-1}\sum_i R\gamma_i$, which is finitely generated, and so is $R[\Delta]$ because $R$ is noetherian.
\end{proof}

\section{A local conjugation theorem}\label{sec:local}

Throughout this section $F$ is a nonarchimedean local field of characteristic zero with ring of integers $R$, maximal ideal $\pfrak$, uniformizer $\varpi$, and finite residue field $\mathbf{F} = R/\pfrak$ of characteristic $p$. We write $v$ for the normalized valuation, $x\mapsto \bar{x}$ for reduction modulo $\pfrak$, and $e = v(2)$, so that $e = 0$ if $p$ is odd and $e = e(\pfrak\mid 2)\geq 1$ if $p = 2$. For $m\geq 1$ let
\begin{align*}
\SL(2,R,\pfrak^m) = \{ g\in \SL(2,R) : g \equiv 1 \bmod \pfrak^m\}
\end{align*}
be the principal congruence subgroup of level $\pfrak^m$, and put $\pm\SL(2,R,\pfrak^m) = \{\pm 1\}\cdot\SL(2,R,\pfrak^m)$. We also use the convention $\pm\SL(2,R,\pfrak^0) = \SL(2,R)$. Note that $-1\in\SL(2,R,\pfrak^m)$ if and only if $m\leq e$. Let $U(\mathbf{F})$ be the group of upper triangular unipotent matrices in $\SL(2,\mathbf{F})$, and let $\slfrak(2,\mathbf{F})$ denote the space of trace zero matrices (which contains the scalars when $p = 2$). For a $2\times 2$ matrix $Y$ we write $Y^\sharp = \tr(Y) - Y$ for its adjugate, so that $YY^\sharp = \det(Y)$ and
\begin{align}\label{eq:det-sum}
\det(Y+Z) = \det Y + \det Z + \tr(Y^\sharp Z).
\end{align}
The proof of Lemma \ref{lemma:congruence-traces} applies verbatim to the order $\M(2,R)$ and shows that
\begin{align}\label{eq:local-congruence-trace}
\tr g - 2\epsilon \in \pfrak^{2m} \qquad \text{for all } g\in\epsilon\SL(2,R,\pfrak^m),\ \epsilon\in\{\pm1\}.
\end{align}

\begin{definition}
Let $n\geq 0$. A subgroup $H\leq\SL(2,R)$ satisfies the \textbf{trace condition} $(\mathrm{T}_n)$ if for every $h\in H$ there is a sign $\epsilon_h\in\{\pm 1\}$ with $\tr h - 2\epsilon_h \in \pfrak^{2n}$.
\end{definition}

Condition $(\mathrm{T}_n)$ is invariant under conjugation by $\GL(2,F)$, it passes to subgroups, and $H$ satisfies it if and only if $\{\pm 1\}H$ does. By \eqref{eq:local-congruence-trace}, every subgroup of $\pm\SL(2,R,\pfrak^n)$ satisfies $(\mathrm{T}_n)$. When $p$ is odd, the two numbers $\tr h \mp 2$ differ by the unit $4$, so $(\mathrm{T}_n)$ is equivalent to $\tr(h)^2 - 4 \in\pfrak^{2n}$ for all $h \in H$. The goal of this section is the following converse.

\begin{theorem}\label{prop:local}
Let $n\geq 0$ and assume that $e\leq 1$ or $n\leq e$. If $H \leq \SL(2,R)$ satisfies $(\mathrm{T}_n)$, then there exists $\alpha\in \GL(2,F)$ such that $\alpha H \alpha^{-1} \subseteq \pm\SL(2,R,\pfrak^n)$.
\end{theorem}

Theorem \ref{prop:local} therefore characterizes, purely in terms of traces, the subgroups of $\SL(2,R)$ that are conjugate into $\pm\SL(2,R,\pfrak^n)$. The sign in the conclusion cannot be removed, since $\{\pm 1\}$ satisfies $(\mathrm{T}_n)$ for every $n$. The restriction on $n$ when $e\geq 2$ is necessary.

\begin{example}\label{ex:dyadic}
 Suppose $e\geq 2$, and let $H = \SL(2,\mathbf{Z}_2,2) \leq \SL(2,R)$, where $\mathbf{Z}_2\subseteq R$. Every $h \in H$ has the form $h = 1+2X$ with $X\in\M(2,\mathbf{Z}_2)$, so $\tr h - 2 = -4\det X$ and $\tr h + 2 = 4(1-\det X)$. One of $\det X$ and $1-\det X$ lies in $2\mathbf{Z}_2\subseteq\pfrak^e\subseteq\pfrak^2$, and $v(4) = 2e$, so $H$ satisfies $(\mathrm{T}_{e+1})$. Suppose that $\alpha H\alpha^{-1} \subseteq \pm\SL(2,R,\pfrak^{e+1})$ and put $L = \alpha^{-1}R^2$. For the unipotent elements $u_1 = \tbt{1}{2}{0}{1}$ and $u_2 = \tbt{1}{0}{2}{1}$ of $H$, the conjugate $\alpha u_i\alpha^{-1}$ cannot lie in $-\SL(2,R,\pfrak^{e+1})$, for then \eqref{eq:local-congruence-trace} would give $4 = \tr u_i + 2 \in \pfrak^{2e+2}$. Hence $(u_i - 1)L\subseteq\pfrak^{e+1}L$ for $i = 1,2$. Since $v(2) = e$, the matrix $J = \tbt{0}{1}{1}{0}$ satisfies $JL\subseteq\pfrak L$, and then $L = J^2L \subseteq\pfrak^2L$, which is absurd. So $H$ satisfies $(\mathrm{T}_{e+1})$ but is not conjugate into $\pm\SL(2,R,\pfrak^{e+1})$.
\end{example}

\subsection{Preliminary lemmas}

\begin{lemma}\label{lemma:sylow}
Let $G\leq \SL(2,\mathbf{F})$ be a subgroup such that $\tr g = \pm 2$ for every $g\in G$. Then there is $g_0\in \SL(2,\mathbf{F})$ with $g_0 G g_0^{-1} \subseteq \{\pm 1\} U(\mathbf{F})$.
\end{lemma}

\begin{proof}
If $\tr g = 2\epsilon$ with $\epsilon=\pm1$, then by the Cayley--Hamilton theorem $(g-\epsilon)^2 = g^2 - 2\epsilon g + 1 = 0$. Hence $g = \epsilon(1+N)$ with $N^2=0$, and $g^p = \epsilon^p(1+pN) = \epsilon^p$. Therefore every element of the image $\bar G$ of $G$ in $\PSL(2,\mathbf{F})$ has order $1$ or $p$, and by Cauchy's theorem $\bar G$ is a $p$--group. Its preimage $P = \{\pm1\}G$ in $\SL(2,\mathbf{F})$ satisfies $P = \{\pm 1\}S$ for a Sylow $p$--subgroup $S$ of $P$. (If $p$ is odd, $S$ has index $2$ in $P$; if $p = 2$, then $-1 = 1$ and $P = S$.) As $|\SL(2,\mathbf{F})| = q(q^2-1)$ with $q = |\mathbf{F}|$, the group $U(\mathbf{F})$ of order $q$ is a Sylow $p$--subgroup of $\SL(2,\mathbf{F})$, so by Sylow's theorems $S$ is conjugate into $U(\mathbf{F})$.
\end{proof}

\begin{lemma}\label{lemma:nilpotent}
Let $V\subseteq \slfrak(2,\mathbf{F})$ be an additive subgroup consisting of nilpotent matrices. Then there is $g_0\in\SL(2,\mathbf{F})$ such that every element of $g_0Vg_0^{-1}$ is strictly upper triangular.
\end{lemma}

\begin{proof}
If $X$ has trace zero, then $X^\sharp = -X$, so \eqref{eq:det-sum} gives $\det(X+Y) = \det X + \det Y - \tr(XY)$. We may assume $V\neq 0$. Pick $0\neq X_0 \in V$. As $X_0$ is nilpotent, $\ker X_0 = \operatorname{im} X_0$ is a line. Let $v_1$ span it and choose $v_2$ with $\det(v_1,v_2) = 1$. With $g_0^{-1} = (v_1\ v_2)$ we get $g_0X_0g_0^{-1} = \tbt{0}{b}{0}{0}$ with $b\neq 0$. Replace $V$ by $g_0Vg_0^{-1}$ and let $Y = \tbt{a}{f}{c}{-a}\in V$. As $X_0$, $Y$ and $X_0+Y$ are nilpotent, we get $0 = -\tr(X_0Y) = -bc$, so $c = 0$. Then $0 = \det Y = -a^2$, so $a = 0$.
\end{proof}

\begin{lemma}\label{lemma:lie-algebra}
Let $m\geq 1$. Every $h\in\SL(2,R,\pfrak^m)$ can be written uniquely as $h = 1+\varpi^mX_h$ with $X_h\in\M(2,R)$, and
\begin{enumerate}
\item $\tr X_h = -\varpi^m\det X_h$ and $\tr h - 2 = -\varpi^{2m}\det X_h$;
\item $X_{hh'} = X_h + X_{h'} + \varpi^m X_hX_{h'}$ and $X_{h^{-1}} = X_h^\sharp$, so that $h \mapsto \overline{X_h}$ is a homomorphism from $\SL(2,R,\pfrak^m)$ to the additive group of $\slfrak(2,\mathbf{F})$ with kernel $\SL(2,R,\pfrak^{m+1})$;
\item $X_{ghg^{-1}} = gX_hg^{-1}$ for every $g \in \GL(2,R)$.
\end{enumerate}
\end{lemma}

\begin{proof}
For any $2\times 2$ matrix $X$ one has $\det(1+\varpi^mX) = 1 + \varpi^m\tr X + \varpi^{2m}\det X$, which gives (1). In particular $\tr X_h \in \pfrak$. The product formula in (2) is immediate, and $h^{-1} = h^\sharp = \tr(h) - h = 1 + \varpi^m(\tr X_h - X_h)$. Statement (3) is clear.
\end{proof}

\subsection{The inductive step}

The next lemma contains the main computation. Let $\alpha_0 = \tbt{\varpi}{0}{0}{1}$.

\begin{lemma}\label{lemma:key}
Let $m\geq 0$, let $H\leq \SL(2,R)$, and let $\epsilon\colon H\to\{\pm 1\}$ be a homomorphism. Suppose that every $h \in H$ can be written as $h = \epsilon(h)(1+\varpi^mX_h)$ with $X_h = (x_{ij})\in \M(2,R)$ such that $x_{11}, x_{22}, x_{21}\in\pfrak$, and that $\tr(\epsilon(h)h) - 2 \in \pfrak^{2m+2}$ for every $h\in H$. Then either $H$ or $\alpha_0H\alpha_0^{-1}$ is contained in $\pm\SL(2,R,\pfrak^{m+1})$.
\end{lemma}

\begin{proof}
 Since $\epsilon$ is a homomorphism, $X_h = \varpi^{-m}(\epsilon(h)h-1)$ satisfies
\begin{align}\label{eq:X-product}
X_{hh'} = X_h + X_{h'} + \varpi^m X_hX_{h'} \qquad (h,h'\in H).
\end{align}

First we show that $x_{12}x_{21} \in \pfrak^2$ for every $h\in H$. Since $1 = \det(\epsilon(h)h) = 1 + \varpi^m\tr X_h + \varpi^{2m}\det X_h$, we have $\tr(\epsilon(h)h) -2 = \varpi^m\tr X_h = -\varpi^{2m}\det X_h$. The hypothesis gives $\det X_h \in \pfrak^2$. As $x_{11}x_{22}\in\pfrak^2$, we get $x_{12}x_{21}\in\pfrak^2$.

Next, we will show that either $x_{12}\in \pfrak$ for all $h\in H$, or $x_{21}\in \pfrak^2$ for all $h\in H$. Suppose not. Then there are $h_1,h_2\in H$ such that $X_1 = X_{h_1}$ has $(X_1)_{12}\notin \pfrak$, and $X_2 = X_{h_2}$ has $(X_2)_{21}\notin\pfrak^2$. By Step 1, $(X_1)_{21}\in\pfrak^2$ and $(X_2)_{12} \in \pfrak$. Let $X_3 = X_{h_1h_2}$, given by \eqref{eq:X-product}. Since all entries of $X_1$ and $X_2$ other than those in position $(1,2)$ lie in $\pfrak$,
\begin{align*}
(X_1X_2)_{12} & = (X_1)_{11}(X_2)_{12} + (X_1)_{12}(X_2)_{22} \in \pfrak, \\
(X_1X_2)_{21} & = (X_1)_{21}(X_2)_{11} + (X_1)_{22}(X_2)_{21} \in \pfrak^2.
\end{align*}
Therefore $(X_3)_{12} \equiv (X_1)_{12} \bmod \pfrak$ is a unit, while $(X_3)_{21} \equiv (X_1)_{21} + (X_2)_{21} \equiv (X_2)_{21} \bmod \pfrak^2$ does not lie in $\pfrak^2$. This contradicts Step 1 for $h_1h_2 \in H$.

If $x_{12}\in\pfrak$ for all $h$, then $X_h \in \pfrak\M(2,R)$ and $h\in \epsilon(h) \SL(2,R,\pfrak^{m+1})$ for all $h$. Otherwise $x_{21}\in\pfrak^2$ for all $h$, and
\begin{align*}
\alpha_0 h\alpha_0^{-1} = \epsilon(h)\left(1+\varpi^m\alpha_0X_h\alpha_0^{-1}\right), \qquad \alpha_0X_h\alpha_0^{-1} = \tbt{x_{11}}{\varpi x_{12}}{\varpi^{-1}x_{21}}{x_{22}} \in \pfrak\M(2,R).
\end{align*}
As $\alpha_0h\alpha_0^{-1}$ has determinant $1$ and integral entries, it lies in $\epsilon(h)\SL(2,R,\pfrak^{m+1})$.
\end{proof}

To apply Lemma \ref{lemma:key} we need a homomorphic choice of signs compatible with the trace condition. This is where the size of $e$ enters.

\begin{lemma}\label{lemma:signs}
Let $m\geq 1$, and assume that $m \neq e$ or $m = e = 1$. Let $\hat H\leq\pm\SL(2,R,\pfrak^m)$ be a subgroup containing $-1$ and satisfying $(\mathrm{T}_{m+1})$. Then
\begin{align*}
H^+ = \{h\in\hat H\cap\SL(2,R,\pfrak^m) : \tr h - 2\in\pfrak^{2m+2}\}
\end{align*}
is a subgroup of $\hat H$ and $\hat H = \{\pm 1\}H^+$.
\end{lemma}

\begin{proof}
Suppose first that $m>e$. Then $-1\notin\SL(2,R,\pfrak^m)$, so $\hat H = \{\pm 1\}\times(\hat H\cap\SL(2,R,\pfrak^m))$. For $h\in\hat H\cap\SL(2,R,\pfrak^m)$ we have $\tr h\equiv 2 \bmod \pfrak^{2m}$ by Lemma \ref{lemma:lie-algebra}(1), and $\tr h\equiv 2\epsilon_h\bmod\pfrak^{2m+2}$. If $\epsilon_h = -1$, then $4\in\pfrak^{2m}$, that is $m\leq e$, a contradiction. Hence $H^+ = \hat H\cap\SL(2,R,\pfrak^m)$.

Suppose next that $m<e$. Then $-1\in\SL(2,R,\pfrak^m)$, so $\hat H\subseteq \SL(2,R,\pfrak^m)$. Since $2m+2\leq 2e = v(4)$, the two numbers $\tr h \mp 2$ are congruent modulo $\pfrak^{2m+2}$, so $(\mathrm{T}_{m+1})$ gives $\tr h - 2\in\pfrak^{2m+2}$ for every $h$. Hence $H^+ = \hat H$.

Finally suppose $m = e = 1$, and put $u = 2/\varpi\in R^\times$. Then $\hat H\subseteq\SL(2,R,\pfrak)$. For $h = 1+\varpi X_h\in\hat H$, Lemma \ref{lemma:lie-algebra}(1) gives
\begin{align*}
\tr h - 2 = -\varpi^2\det X_h, \qquad \tr h + 2 = \varpi^2\left(u^2 - \det X_h\right).
\end{align*}
Thus $(\mathrm{T}_2)$ says that $\det X_h\equiv 0$ or $\det X_h \equiv u^2 \bmod\pfrak^2$, and $h\in H^+$ exactly when $\det X_h \in \pfrak^2$. As $\tr(-h) = -\tr h$, exactly one of $h$ and $-h$ lies in $H^+$. Since $\tr h^{-1} = \tr h$, the set $H^+$ is closed under inversion, and it remains to show that it is closed under products. Let $h_1, h_2\in H^+$ and $X_i = X_{h_i}$. Then $\det X_i\in\pfrak^2$, $\tr X_i = -\varpi\det X_i\in\pfrak^3$, and $\overline{X_i}$ is nilpotent. By Lemma \ref{lemma:lie-algebra}(2), $h_2^{-1} = 1+\varpi X_2^\sharp$, and $X_2^\sharp$ has the same three properties.

We claim that if $X'\in\{X_2, X_2^\sharp\}$ and $h' = 1+\varpi X'$, then $\det X_{h_1h'}\equiv -\tr(X_1X') \bmod\pfrak^2$. By \eqref{eq:det-sum},
\begin{align*}
\det(X_1+X') = \det X_1 + \det X' + \tr X_1 \tr X' - \tr(X_1X') \equiv -\tr(X_1X') \bmod \pfrak^2.
\end{align*}
Moreover $X_{h_1h'} = (X_1+X') + \varpi X_1X'$, and by \eqref{eq:det-sum} again,
\begin{align*}
\det X_{h_1h'} - \det(X_1+X') = \varpi\left(\tr(X_1+X')\tr(X_1X') - \tr(X_1^2X') - \tr(X_1X'^2)\right) + \varpi^2\det(X_1X').\end{align*}
The bracket lies in $\pfrak$, because $\tr(X_1+X')\in\pfrak$ and $\overline{X_1}^2 = \overline{X'}^2 = 0$. This proves the claim.

Now $\tr(X_1X_2^\sharp) = \tr X_1\tr X_2 - \tr(X_1X_2)\equiv -\tr(X_1X_2)\bmod\pfrak^2$, so the claim gives
\begin{align*}
\det X_{h_1h_2}\equiv -\tr(X_1X_2), \qquad \det X_{h_1h_2^{-1}}\equiv \tr(X_1X_2) \pmod{\pfrak^2}.
\end{align*}
If $h_1h_2\notin H^+$, then $\det X_{h_1h_2}\equiv u^2$, hence $\det X_{h_1h_2^{-1}}\equiv -u^2\bmod\pfrak^2$. By $(\mathrm{T}_2)$ for $h_1h_2^{-1}$, either $u^2\in\pfrak^2$, which is false, or $2u^2\in\pfrak^2$, which contradicts $v(2) = 1$. Hence $h_1h_2\in H^+$.
\end{proof}

\begin{remark}
Lemma \ref{lemma:signs} cannot hold for $m = e\geq 2$. Otherwise the proof of Lemma \ref{lemma:one-step} below would conjugate the group of Example \ref{ex:dyadic}, which lies in $\SL(2,R,\pfrak^e)$, into $\pm\SL(2,R,\pfrak^{e+1})$. In the last case of the proof above, the final contradiction requires $2u^2\notin\pfrak^2$, which holds exactly when $e = 1$.
\end{remark}

\begin{lemma}\label{lemma:one-step}
Let $m\geq 0$, and assume that $m\neq e$ or $e\leq 1$. Let $H\leq\pm\SL(2,R,\pfrak^m)$ be a subgroup satisfying $(\mathrm{T}_{m+1})$. Then there exists $\alpha\in\GL(2,F)$ with $\alpha H \alpha^{-1}\subseteq\pm\SL(2,R,\pfrak^{m+1})$.
\end{lemma}

\begin{proof}
Recall that reduction $\SL(2,R)\to\SL(2,\mathbf{F})$ is surjective, since $\SL(2,\mathbf{F})$ is generated by elementary matrices. Suppose $m=0$. For $h\in H$ we have $\tr\bar h = 2\bar\epsilon_h = \pm 2$. By Lemma \ref{lemma:sylow} there is $g\in\SL(2,R)$ such that every element of $gHg^{-1}$ reduces into $\{\pm 1\}U(\mathbf{F})$, and we replace $H$ by $gHg^{-1}$. If $p$ is odd, then $\{\pm 1\}U(\mathbf{F}) = \{\pm1\}\times U(\mathbf{F})$, and we let $\epsilon(h)$ be the sign with $\bar h\in\epsilon(h)U(\mathbf{F})$. If $p = 2$, we put $\epsilon(h) = 1$. In both cases $\epsilon$ is a homomorphism, and $h = \epsilon(h)(1+X_h)$ with $X_h = \epsilon(h)h-1$, whose entries $(1,1)$, $(2,2)$ and $(2,1)$ lie in $\pfrak$. By $(\mathrm{T}_1)$ there is a sign $\epsilon'$ with $\tr(\epsilon(h)h) - 2\epsilon'\in\pfrak^2$. If $p$ is odd, then $\tr(\epsilon(h)h)\equiv 2\bmod\pfrak$ forces $\epsilon' = 1$. If $p=2$, then $2 - 2\epsilon'\in\{0,4\}\subseteq\pfrak^2$. In both cases $\tr(\epsilon(h)h) - 2\in\pfrak^2$, and Lemma \ref{lemma:key} applies.

Suppose $m\geq 1$. The group $\hat H = \{\pm 1\}H$ satisfies the hypotheses of Lemma \ref{lemma:signs}. Let $H^+$ be the subgroup it provides. For $h\in H^+$ write $h = 1+\varpi^mX_h$. By Lemma \ref{lemma:lie-algebra}(1), $\tr h - 2 = -\varpi^{2m}\det X_h$, so $\det X_h\in\pfrak^2$ and $\tr X_h = -\varpi^m\det X_h \in \pfrak$. Thus $\overline{X_h}$ is nilpotent, and by Lemma \ref{lemma:lie-algebra}(2) the set $V = \{\overline{X_h} : h\in H^+\}$ is an additive subgroup of $\slfrak(2,\mathbf{F})$. Lemma \ref{lemma:nilpotent} provides $g\in\SL(2,R)$ such that $\overline{gX_hg^{-1}}$ is strictly upper triangular for all $h\in H^+$. By Lemma \ref{lemma:lie-algebra}(3), $gH^+g^{-1}$ satisfies the hypotheses of Lemma \ref{lemma:key} with $\epsilon\equiv 1$. Hence there is $\alpha\in\{g, \alpha_0 g\}$ with $\alpha H^+\alpha^{-1}\subseteq\pm\SL(2,R,\pfrak^{m+1})$, and then $\alpha H\alpha^{-1}\subseteq \alpha\hat H\alpha^{-1} = \{\pm 1\}\alpha H^+\alpha^{-1} \subseteq\pm\SL(2,R,\pfrak^{m+1})$.
\end{proof}

\begin{proof}[Proof of Theorem \ref{prop:local}]
We argue by induction on $n$, the case $n=0$ being trivial. Let $n \geq 0$, suppose that $e\leq 1$ or $n+1\leq e$, and let $H$ satisfy $(\mathrm{T}_{n+1})$. Then $H$ satisfies $(\mathrm{T}_n)$, and the hypothesis of the proposition holds for $n$, so there is $\alpha$ with $\alpha H\alpha^{-1}\subseteq\pm\SL(2,R,\pfrak^n)$. If $e\leq 1$, or if $n+1\leq e$ so that $n\neq e$, then Lemma \ref{lemma:one-step} applies to $\alpha H\alpha^{-1}$ with $m=n$ and provides $\alpha'$ with $\alpha'\alpha H\alpha^{-1}\alpha'^{-1}\subseteq\pm\SL(2,R,\pfrak^{n+1})$.
\end{proof}

\section{Arithmetic Fuchsian and Kleinian groups}\label{sec:arithmetic-groups}

\subsection{Hyperbolic geometry}

The group of orientation preserving isometries of $\mathbf{H}^2$ is $\PSL(2,\mathbf{R})$ and that of $\mathbf{H}^3$ is $\PSL(2,\mathbf{C})$. An element $g\in\GL(2,\kinf)$ acts on $\mathbf{H}^n$ through its image in $\PGL(2,\kinf)$ by an isometry, which reverses orientation when $\kinf = \mathbf{R}$ and $\det g<0$. Consequently, if $\Lambda$ is a discrete subgroup of $\PSL(2,\kinf)$ and $g\in\GL(2,\kinf)$, then $\Lambda\backslash\mathbf{H}^n$ and $g\Lambda g^{-1}\backslash\mathbf{H}^n$ are isometric. For a subgroup $\Lambda$ of $\PSL(2,\kinf)$ we write $\tilde\Lambda$ for its full preimage in $\SL(2,\kinf)$. Thus $-1\in\tilde\Lambda$ and $\tr\tilde\Lambda = -\tr\tilde\Lambda$.

Let $\Lambda\leq\PSL(2,\kinf)$ be a torsion free cocompact lattice. Every nontrivial element of $\Lambda$ is hyperbolic (if $n=2$) or loxodromic (if $n=3$), and its conjugacy class corresponds to a closed geodesic of $\Lambda\backslash\mathbf{H}^n$. If $\gamma \in \tilde\Lambda$ is conjugate to $\pm\tbt{e^{\lambda/2}}{0}{0}{e^{-\lambda/2}}$ with $\operatorname{Re}\lambda>0$, then $\lambda$, which is well defined modulo $2\pi i\mathbf{Z}$, is the \textbf{complex length} of $\gamma$ and $\ell(\gamma) = \operatorname{Re}\lambda$ is the length of the corresponding closed geodesic. In particular
\begin{align}\label{eq:tracelength}
\tr\gamma = \pm 2\cosh(\lambda/2), \qquad \text{and, if } n = 2,\quad |\tr\gamma| = 2\cosh(\ell(\gamma)/2).
\end{align}
Two loxodromic elements of $\SL(2,\mathbf{C})$ have the same complex length if and only if their traces agree up to sign.

\subsection{Groups derived from quaternion algebras}

Let $k$ be a number field with $r_1$ real and $r_2$ complex places and let $B$ be a quaternion algebra over $k$ ramified at $s$ real places. Then \cite[Theorem 8.1.1]{Maclachlan.Reid-[ArithmeticHyperbolic3Manifolds]2003}
\begin{align}\label{eqn:decomp}
B \otimes_\mathbf{Q} \mathbf{R} \cong \mathcal{H}^{s} \oplus \M(2,\mathbf{R})^{r_1-s} \oplus \M(2,\mathbf{C})^{r_2}.
\end{align}
Assume that $B$ satisfies the Eichler condition. If $\Ocal$ is an order of $B$, the image of $\Ocal^1$ under the projection of \eqref{eqn:decomp} to the noncompact factors is a lattice in $\SL(2,\mathbf{R})^{r_1-s}\times\SL(2,\mathbf{C})^{r_2}$, and it is cocompact if and only if $B$ is a division algebra \cite[Theorem 8.1.2]{Maclachlan.Reid-[ArithmeticHyperbolic3Manifolds]2003}. We are interested in the two cases in which there is a single noncompact factor, namely algebras of Fuchsian and Kleinian type as defined in the introduction. Let $\nu_0$ be the unramified archimedean place. We fix once and for all the embedding $k\hookrightarrow\kinf$ given by $\nu_0$ and identify $k$ with its image, and we fix an embedding of $k$--algebras $\psi\colon B\to\M(2,\kinf)$ inducing $B_{\nu_0}\cong\M(2,\kinf)$. Any other choice differs by conjugation by an element of $\GL(2,\kinf)$, by the Skolem--Noether theorem. Note that $\trd = \tr\circ\psi$ and $\nrd = \det\circ\psi$. For an order $\Ocal$ and an ideal $\afrak$ we write $\Gamma_\Ocal(\afrak)$ for the image of $\psi(\Ocal^1(\afrak))$ in $\PSL(2,\kinf)$ and put $\Gamma_\Ocal = \Gamma_\Ocal(R)$. Then $\Gamma_\Ocal$ is a cocompact lattice when $B$ is a division algebra, $\Gamma_\Ocal(\afrak)$ is a normal subgroup of finite index, and $\tilde\Gamma_\Ocal(\afrak) = \pm\psi(\Ocal^1(\afrak))$. A subgroup of $\PSL(2,\kinf)$ is an \textbf{arithmetic Fuchsian} (respectively \textbf{Kleinian}) \textbf{group} if it is commensurable in the wide sense with some $\Gamma_\Ocal$ as above, and it is \textbf{derived from a quaternion algebra} if it is a finite index subgroup of some $\Gamma_\Ocal$.

If $\beta\in B^\times$, then $\psi(\beta)\Gamma_\Ocal(\afrak)\psi(\beta)^{-1} = \Gamma_{\beta\Ocal\beta^{-1}}(\afrak)$, so the isometry class of $X_\Ocal(\afrak)$ depends only on the conjugacy class of $\Ocal$. Different maximal orders give lattices of the same covolume.

\begin{lemma}\label{lemma:covolume}
Let $B$ be of Fuchsian or Kleinian type, let $\Ocal$ and $\Ocal'$ be maximal orders of $B$, and let $\afrak$ be a nonzero ideal of $R$. Then $\Gamma_\Ocal(\afrak)$ and $\Gamma_{\Ocal'}(\afrak)$ have the same covolume.
\end{lemma}

\begin{proof}
Let $G = B^1$, viewed as an algebraic group over $k$, and let $G(\mathbf{A}_f)$ be the restricted product of the groups $B_\pfrak^1$. Put $K_\pfrak = \Ocal_\pfrak^1\cap(1+\afrak\Ocal_\pfrak)$ and $K'_\pfrak = \Ocal'^1_\pfrak\cap(1+\afrak\Ocal'_\pfrak)$, and let $U = \prod_\pfrak K_\pfrak$ and $U' = \prod_\pfrak K'_\pfrak$. These are compact open subgroups of $G(\mathbf{A}_f)$, and by Proposition \ref{prop:local-global-lattices} and \eqref{eq:aO-local} we have $\Ocal^1(\afrak) = G(k)\cap U$ and $\Ocal'^1(\afrak) = G(k)\cap U'$. Since $B$ is unramified at an archimedean place, strong approximation \cite[Theorem 7.7.5]{Maclachlan.Reid-[ArithmeticHyperbolic3Manifolds]2003} says that $G(k)$ is dense in $G(\mathbf{A}_f)$. Hence $\Ocal^1(\afrak)$ is dense in $U$, and as $U\cap U'$ is open, the natural map
\begin{align*}
\Ocal^1(\afrak)\big/\left(\Ocal^1(\afrak)\cap\Ocal'^1(\afrak)\right) \longrightarrow U/(U\cap U')
\end{align*}
is a bijection. Thus $[\Ocal^1(\afrak):\Ocal^1(\afrak)\cap\Ocal'^1(\afrak)] = [U:U\cap U']$, and similarly with the roles of $\Ocal$ and $\Ocal'$ exchanged. Let $\mu = \prod_\pfrak\mu_\pfrak$ be a Haar measure on $G(\mathbf{A}_f)$. Then $[U:U\cap U'] / [U':U\cap U'] = \mu(U)/\mu(U')$, and it suffices to show $\mu_\pfrak(K_\pfrak) = \mu_\pfrak(K'_\pfrak)$ for every $\pfrak$.

By Proposition \ref{prop:local-orders} there is $g_\pfrak\in B_\pfrak^\times$ with $\Ocal'_\pfrak = g_\pfrak\Ocal_\pfrak g_\pfrak^{-1}$, so $K'_\pfrak = g_\pfrak K_\pfrak g_\pfrak^{-1}$. Conjugation by $g\in B_\pfrak^\times$ is a continuous automorphism of $B_\pfrak^1$, so it multiplies $\mu_\pfrak$ by a positive number $\Delta(g)$, and $\Delta\colon B_\pfrak^\times\to\mathbf{R}_{>0}$ is a homomorphism. It is trivial on $k_\pfrak^\times B_\pfrak^1$, because $B_\pfrak^1$ is unimodular. As $\nrd$ maps $B_\pfrak^\times/k_\pfrak^\times B_\pfrak^1$ injectively into the finite group $k_\pfrak^\times/(k_\pfrak^\times)^2$, the image of $\Delta$ is a finite subgroup of $\mathbf{R}_{>0}$, hence trivial. Therefore $\mu_\pfrak(K'_\pfrak) = \mu_\pfrak(K_\pfrak)$. We conclude that $[\Ocal^1(\afrak):\Ocal^1(\afrak)\cap\Ocal'^1(\afrak)] = [\Ocal'^1(\afrak):\Ocal^1(\afrak)\cap\Ocal'^1(\afrak)]$, so the lattices $\psi(\Ocal^1(\afrak))$ and $\psi(\Ocal'^1(\afrak))$ of $\SL(2,\kinf)$ have the same covolume. Finally $-1\in\Ocal^1(\afrak)$ if and only if $2\in\afrak\Ocal\cap k = \afrak$, and likewise for $\Ocal'$, so the images in $\PSL(2,\kinf)$ also have the same covolume.
\end{proof}

\subsection{Invariants}

Let $\Lambda\leq\PSL(2,\kinf)$ be a non-elementary finitely generated group. Its \textbf{trace field} is $\mathbf{Q}(\tr\tilde\Lambda)$, and the $\mathbf{Q}(\tr\tilde\Lambda)$--span $A_0\tilde\Lambda$ of $\tilde\Lambda$ in $\M(2,\kinf)$ is a quaternion algebra over $\mathbf{Q}(\tr\tilde\Lambda)$ \cite[Theorem 3.2.1]{Maclachlan.Reid-[ArithmeticHyperbolic3Manifolds]2003}. Let $\Lambda^{(2)}\leq\Lambda$ be the subgroup generated by squares, and let $\tilde\Lambda^{(2)}\leq\SL(2,\kinf)$ be the subgroup generated by the squares of elements of $\tilde\Lambda$. Since $(-\lambda)^2 = \lambda^2$, the image of $\tilde\Lambda^{(2)}$ in $\PSL(2,\kinf)$ is $\Lambda^{(2)}$. The group $\Lambda/\Lambda^{(2)}$ is finitely generated of exponent $2$, hence finite. The \textbf{invariant trace field} $k\Lambda = \mathbf{Q}(\tr\tilde\Lambda^{(2)})$ and the \textbf{invariant quaternion algebra} $A\Lambda = A_0\tilde\Lambda^{(2)}$ are invariants of the wide commensurability class of $\Lambda$ \cite[\S 3.3]{Maclachlan.Reid-[ArithmeticHyperbolic3Manifolds]2003}. If $B$ is of Fuchsian or Kleinian type and $\Ocal$ is an order of $B$, then $k\Gamma_\Ocal = k$ and $A\Gamma_\Ocal \cong B$ \cite[Chapter 8]{Maclachlan.Reid-[ArithmeticHyperbolic3Manifolds]2003}.

Arithmeticity can be detected from traces. The following is Takeuchi's characterization \cite{Takeuchi-[CharacterizationArithmeticFuchsian]1975}, together with its Kleinian analogue \cite[\S 8.3]{Maclachlan.Reid-[ArithmeticHyperbolic3Manifolds]2003}.

\begin{theorem}\label{thm:takeuchi}
Let $\Lambda\leq\PSL(2,\kinf)$ be a lattice.
\begin{enumerate}
\item If $\kinf = \mathbf{R}$, then $\Lambda$ is derived from a quaternion algebra if and only if $\mathbf{Q}(\tr\tilde\Lambda)$ is a number field, $\tr\tilde\Lambda$ consists of algebraic integers, and $\sigma(\tr\tilde\Lambda)$ is bounded for every embedding $\sigma\neq\mathrm{id}$ of $\mathbf{Q}(\tr\tilde\Lambda)$ into $\mathbf{C}$.
\item If $\kinf = \mathbf{C}$, then $\Lambda$ is arithmetic if and only if $k\Lambda$ has exactly one complex place, $\tr\tilde\Lambda$ consists of algebraic integers, and $A\Lambda$ is ramified at every real place of $k\Lambda$. The last condition holds if $\sigma(\tr \tilde\Lambda^{(2)})$ is bounded for every real embedding $\sigma$ of $k\Lambda$.
\end{enumerate}
\end{theorem}

The sufficient condition in (2) is elementary. If $A\Lambda$ were split at a real place $\sigma$, then $\sigma(\tilde\Lambda^{(2)})$ would be a subgroup of $\SL(2,\mathbf{R})$ that spans $A\Lambda\otimes_\sigma\mathbf{R}\cong\M(2,\mathbf{R})$. If its traces are bounded, it contains no hyperbolic elements, since $|\tr g^n|\to\infty$ for hyperbolic $g$. A subgroup of $\PSL(2,\mathbf{R})$ without hyperbolic elements fixes a point of $\mathbf{H}^2$ or of $\partial\mathbf{H}^2$, by a standard argument. It is then conjugate into $\mathrm{SO}(2)$ or into the upper triangular group, and in either case it spans a proper subalgebra of $\M(2,\mathbf{R})$, a contradiction.

We shall use the following commensurability theorems. The first is due to Reid \cite{Reid-[IsospectralityCommensurabilityArithmetic]1992} and the second to Chinburg, Hamilton, Long and Reid \cite[Theorem 1.1 and Corollary 1.2]{Chinburg.Hamilton.Long.Reid-[GeodesicsCommensurabilityClasses]2008}.

\begin{theorem}\label{thm:reid}
Let $M_1$ and $M_2$ be closed orientable hyperbolic surfaces with the same length spectrum. If $M_1$ is arithmetic, then $M_1$ and $M_2$ are commensurable.
\end{theorem}

\begin{theorem}\label{thm:chlr}
Let $M_1$ and $M_2$ be closed arithmetic hyperbolic $3$--manifolds. If $M_1$ and $M_2$ are isospectral, then they are commensurable.
\end{theorem}

\section{Spectral preliminaries}\label{sec:spectral}

Let $M$ be a closed Riemannian manifold of dimension $d$ with scalar curvature $s$, Ricci tensor $\operatorname{Ric}$ and curvature tensor $\operatorname{Rm}$. The heat trace has an asymptotic expansion $\sum_j e^{-\lambda_j t}\sim (4\pi t)^{-d/2}\sum_{i\geq 0} a_i(M) t^{i}$ as $t\to 0^+$, where
\begin{align*}
a_0 = \vol(M),\qquad a_1 = \frac{1}{6}\int_M s,\qquad a_2 = \frac{1}{360}\int_M \left(2|\operatorname{Rm}|^2 - 2|\operatorname{Ric}|^2 + 5 s^2\right),
\end{align*}
see \cite{Berger.Gauduchon.Mazet-[SpectreVarieteRiemannienne]1971}. Isospectral manifolds therefore have the same dimension, volume and coefficients $a_i$.

\begin{proposition}[{\cite[\S E.IV]{Berger.Gauduchon.Mazet-[SpectreVarieteRiemannienne]1971}, \cite{Tanno-[EigenvaluesLaplacianRiemannian]1973}}]\label{prop:curvature}
Let $X$ be a closed hyperbolic manifold of dimension $d\in\{2,3\}$ and let $M$ be a closed Riemannian manifold isospectral to $X$. Then $M$ has dimension $d$, constant sectional curvature $-1$, and $\vol(M) = \vol(X)$.
\end{proposition}

\begin{proof}
Write $V$ for the common volume. If $d=2$, then $|\operatorname{Rm}|^2 = s^2$ and $|\operatorname{Ric}|^2 = s^2/2$, so $a_2 = \frac{1}{60}\int s^2$. For $X$ we have $s\equiv -2$, hence $\int_M s = -2V$ and $\int_M s^2 = 4V$. Then $(\int_M s)^2 = V\int_M s^2$, which is the equality case of the Cauchy--Schwarz inequality, so $s$ is constant on $M$, equal to $-2$. If $d = 3$, the Weyl tensor vanishes and $|\operatorname{Rm}|^2 = 4|\operatorname{Ric}|^2 - s^2$. Writing $\operatorname{Ric} = E + \frac{s}{3}g$ with $E$ trace free, the integrand of $a_2$ becomes $6|E|^2 + 5s^2$. For $X$ we have $E = 0$ and $s\equiv -6$, so $\int_M s = -6V$ and $\int_M(6|E|^2+5s^2) = 180V$. By Cauchy--Schwarz $\int_M s^2 \geq 36V$, whence $\int_M |E|^2 \leq 0$. Thus $E=0$ and equality holds in Cauchy--Schwarz, so $M$ is Einstein with $s \equiv -6$. In dimension three this means constant sectional curvature $-1$.
\end{proof}

If $M$ is in addition orientable, then by the Killing--Hopf theorem $M = \Lambda\backslash\mathbf{H}^d$ for a torsion free cocompact lattice $\Lambda\leq\PSL(2,\kinf)$. For a closed hyperbolic manifold $M$ we write $\LS(M)$ for its \textbf{length set}, the set of lengths of closed geodesics, and, when $d = 3$, $\CL(M)$ for the set of their complex lengths. The \textbf{length spectrum} is the length set counted with multiplicities, and similarly for the complex length spectrum.

\begin{theorem}[Huber \cite{Huber-[ZurAnalytischenTheorie]1959}, see also {\cite[Chapter 9]{Buser-[GeometrySpectraCompact]1992}}]\label{thm:huber}
Two closed orientable hyperbolic surfaces are isospectral if and only if they have the same length spectrum.
\end{theorem}

In dimension $3$ the Laplace spectrum determines the length set \cite{Duistermaat.Guillemin-[SpectrumPositiveElliptic]1975}, but the relation with the complex length spectrum is more subtle. The complex length spectrum and the representation spectrum of a closed hyperbolic $3$--manifold determine each other \cite{Salvai-[LaplaceComplexLength]2002}, and the representation spectrum determines the Laplace spectrum \cite{Kelmer-[RefinementStrongMultiplicity]2014}.

\section{Proof of the main theorems}\label{sec:main}

The following statement contains the arithmetic content of Theorems \ref{thm:A} and \ref{thm:B}.

\begin{proposition}\label{prop:key}
 Assume \textup{(H1)} and \textup{(H2)}. Let $\Lambda\leq\PSL(2,\kinf)$ be a lattice such that
\begin{enumerate}
\item[(i)] for every $\lambda\in\tilde\Lambda$ there is $\gamma\in\psi(\Ocal^1(\afrak))$ with $\tr\lambda = \pm\tr\gamma$,
\item[(ii)] $\Lambda$ is commensurable in the wide sense with $\Gamma_\Ocal$, and
\item[(iii)] $\vol(\Lambda\backslash\mathbf{H}^n) = \vol(\Gamma_\Ocal(\afrak)\backslash\mathbf{H}^n)$.
\end{enumerate}
Then there are a maximal order $\Ocal'$ of $B$ and $g\in\GL(2,\kinf)$ with $g\Lambda g^{-1} = \Gamma_{\Ocal'}(\afrak)$.
\end{proposition}

\begin{proof}
Write $\afrak = \prod_{\pfrak}\pfrak^{n_\pfrak}$.

\emph{Step 1: the quaternion algebra of $\tilde\Lambda$.} By (i) and Lemma \ref{lemma:congruence-traces}, $\tr\tilde\Lambda\subseteq R$, so $F = \mathbf{Q}(\tr\tilde\Lambda)$ is a subfield of $k$. By (ii) and the invariance of $k\Lambda$ and $A\Lambda$ under commensurability, $k\Lambda = k\Gamma_\Ocal = k$ and $A\Lambda\cong B$. Since $k\Lambda\subseteq F\subseteq k$, we get $F = k$. Both $A\Lambda = A_0\tilde\Lambda^{(2)}$ and $A_0\tilde\Lambda$ are quaternion algebras over $k$ and the first is contained in the second, so they are equal. Hence there is an isomorphism of $k$--algebras $\theta\colon A_0\tilde\Lambda\to B$.

\emph{Step 2: conjugating $\tilde\Lambda$ into $\psi(B^1)$.} The inclusion $A_0\tilde\Lambda\subseteq\M(2,\kinf)$ and $\psi\circ\theta$ are two $k$--algebra embeddings of $A_0\tilde\Lambda$ into $\M(2,\kinf)$. Both extend to isomorphisms $A_0\tilde\Lambda\otimes_k\kinf\to\M(2,\kinf)$ of central simple $\kinf$--algebras, so by the Skolem--Noether theorem there is $g_1\in\GL(2,\kinf)$ with $\psi(\theta(x)) = g_1xg_1^{-1}$ for all $x\in A_0\tilde\Lambda$. Since $\nrd\circ\theta = \det$, we get $g_1\tilde\Lambda g_1^{-1}\subseteq\psi(B^1)$. Conditions (i)--(iii) are invariant under conjugation, so we may replace $\Lambda$ by $g_1\Lambda g_1^{-1}$ and identify $\tilde\Lambda$ with the subgroup $\Delta = \psi^{-1}(\tilde\Lambda)$ of $B^1$.

\emph{Step 3: a maximal order containing $\Delta$.} The group $\Delta$ spans $B$ by Step 1 and has traces in $R$, so by Lemma \ref{lemma:order-generated} there is a maximal order $\Ocal_1$ with $\Delta\leq\Ocal_1^1$.

\emph{Step 4: local conjugation.} Let $\lambda\in\Delta$. By (i) there are $\epsilon\in\{\pm 1\}$ and $\gamma\in\Ocal^1(\afrak)$ with $\trd\lambda = \epsilon\trd\gamma$, and Lemma \ref{lemma:congruence-traces} gives
\begin{align}\label{eq:global-sign}
\trd\lambda - 2\epsilon = \epsilon(\trd\gamma - 2) \in\afrak^2.
\end{align}
Let $\pfrak\mid\afrak$. By (H2), $B$ is unramified at $\pfrak$, so by Proposition \ref{prop:local-orders} there is an isomorphism $B_\pfrak\cong\M(2,k_\pfrak)$ carrying $\Ocal_{1,\pfrak}$ onto $\M(2,R_\pfrak)$. It carries $\Ocal^1_{1,\pfrak}\cap\pm(1+\pfrak^{m}\Ocal_{1,\pfrak})$ onto $\pm\SL(2,R_\pfrak,\pfrak^m)$ for every $m$. By \eqref{eq:global-sign}, the image of $\Delta$ satisfies $(\mathrm{T}_{n_\pfrak})$, and by (H2) either $v_\pfrak(2)\leq 1$ or $n_\pfrak\leq v_\pfrak(2)$. Theorem \ref{prop:local} therefore provides $\alpha_\pfrak\in B_\pfrak^\times$ with
\begin{align}\label{eq:local-conj}
\alpha_\pfrak\Delta\alpha_\pfrak^{-1}\subseteq\pm\left(1+\pfrak^{n_\pfrak}\Ocal_{1,\pfrak}\right).
\end{align}

\emph{Step 5: globalization.} By Proposition \ref{prop:local-global-lattices} there is a maximal order $\Ocal'$ of $B$ with $\Ocal'_\pfrak = \alpha_\pfrak^{-1}\Ocal_{1,\pfrak}\alpha_\pfrak$ for $\pfrak\mid\afrak$ and $\Ocal'_\mathfrak{q} = \Ocal_{1,\mathfrak{q}}$ for $\mathfrak{q}\nmid\afrak$. Fix $\lambda\in\Delta$. If $\mathfrak{q}\nmid\afrak$, then $\lambda\in\Ocal_{1,\mathfrak{q}} = \Ocal'_\mathfrak{q}$. If $\pfrak\mid\afrak$, then by \eqref{eq:local-conj} there is $\epsilon_\pfrak\in\{\pm1\}$ with
\begin{align*}
\lambda\in\epsilon_\pfrak\,\alpha_\pfrak^{-1}\left(1 + \pfrak^{n_\pfrak}\Ocal_{1,\pfrak}\right)\alpha_\pfrak = \epsilon_\pfrak\left(1+\pfrak^{n_\pfrak}\Ocal'_\pfrak\right).
\end{align*}
By Proposition \ref{prop:local-global-lattices}(1), $\lambda\in\Ocal'$, hence $\lambda\in\Ocal'^1$. The signs are coherent: we claim that $\lambda\in\epsilon(1+\pfrak^{n_\pfrak}\Ocal'_\pfrak)$ for every $\pfrak\mid\afrak$, with the sign $\epsilon$ of \eqref{eq:global-sign}. If $n_\pfrak\leq v_\pfrak(2)$, then $-1\in 1+\pfrak^{n_\pfrak}\Ocal'_\pfrak$, so the two cosets $\pm(1+\pfrak^{n_\pfrak}\Ocal'_\pfrak)$ coincide and there is nothing to prove. Otherwise, the local form \eqref{eq:local-congruence-trace} of Lemma \ref{lemma:congruence-traces} gives $\trd \lambda\equiv 2\epsilon_\pfrak\bmod\pfrak^{2n_\pfrak}$, while $\trd\lambda\equiv 2\epsilon\bmod\pfrak^{2n_\pfrak}$ by \eqref{eq:global-sign}. If $\epsilon_\pfrak\neq\epsilon$, then $4\in\pfrak^{2n_\pfrak}$, that is $n_\pfrak\leq v_\pfrak(2)$, a contradiction. This proves the claim. Therefore $\lambda - \epsilon\in\Ocal'$ lies in $\pfrak^{n_\pfrak}\Ocal'_\pfrak$ for each $\pfrak\mid\afrak$, and by \eqref{eq:aO-local} $\lambda-\epsilon\in\afrak\Ocal'$, that is $\lambda\in\pm\Ocal'^1(\afrak)$. We have shown that $\Delta\subseteq\pm\Ocal'^1(\afrak)$, so $\Lambda\leq\Gamma_{\Ocal'}(\afrak)$.

\emph{Step 6: equality.} By (iii) and Lemma \ref{lemma:covolume}, $\Lambda$ and $\Gamma_{\Ocal'}(\afrak)$ have the same covolume, so $[\Gamma_{\Ocal'}(\afrak):\Lambda] = 1$.
\end{proof}

\begin{remark}\label{rem:type-number}
The maximal order $\Ocal'$ is produced by local patching, and nothing forces it to be conjugate to $\Ocal$. This is exactly the freedom exploited by Vign\'eras. If $t(B) = 1$, then $\Ocal' = \beta^{-1}\Ocal\beta$ for some $\beta\in B^\times$, and $\psi(\beta)g\Lambda g^{-1}\psi(\beta)^{-1} = \Gamma_\Ocal(\afrak)$.
\end{remark}

\subsection{Proof of Theorem \ref{thm:A}}

By (H1) and (H3), $\Gamma_\Ocal(\afrak)$ is a torsion free cocompact lattice in $\PSL(2,\mathbf{R})$, so $X = X_\Ocal(\afrak)$ is a closed orientable hyperbolic surface. It is arithmetic by definition. Let $M$ be a closed orientable Riemannian manifold isospectral to $X$. By Proposition \ref{prop:curvature} and the Killing--Hopf theorem, $M = \Lambda\backslash\mathbf{H}^2$ for a torsion free cocompact lattice $\Lambda\leq\PSL(2,\mathbf{R})$, and $\vol(M) = \vol(X)$. By Theorem \ref{thm:huber}, $M$ and $X$ have the same length spectrum, so by Theorem \ref{thm:reid} the surface $M$ is commensurable with $X$, hence with $\Gamma_\Ocal\backslash\mathbf{H}^2$. Let $\lambda\in\tilde\Lambda$. If $\lambda = \pm1$, condition (i) of Proposition \ref{prop:key} holds with $\gamma = 1$. Otherwise $\lambda$ is hyperbolic, and as $\LS(M) = \LS(X)$ there is a hyperbolic $\gamma\in\psi(\Ocal^1(\afrak))$ with $\ell(\gamma) = \ell(\lambda)$. By \eqref{eq:tracelength}, $|\tr\lambda| = |\tr\gamma|$, that is $\tr\lambda = \pm\tr\gamma$. Proposition \ref{prop:key} now provides a maximal order $\Ocal'$ such that $\Lambda$ is conjugate to $\Gamma_{\Ocal'}(\afrak)$ by an element of $\GL(2,\mathbf{R})$, so $M$ is isometric to $X_{\Ocal'}(\afrak)$. As $X_{\Ocal'}(\afrak)$ depends up to isometry only on the conjugacy class of $\Ocal'$, the spectral genus of $X$ has at most $t(B)$ elements. \qed

\subsection{Proof of Theorem \ref{thm:B}}

As before, $X = X_\Ocal(\afrak)$ is a closed orientable arithmetic hyperbolic $3$--manifold. If $M$ is isospectral to $X$, then by Proposition \ref{prop:curvature} it is a closed hyperbolic $3$--manifold with $\vol(M) = \vol(X)$, and $M = \Lambda\backslash\mathbf{H}^3$ with $\Lambda\leq\PSL(2,\mathbf{C})$ torsion free and cocompact. This is the first assertion.

Assume now that $\CL(M)\subseteq\CL(X)$. Every $\lambda\in\tilde\Lambda$ other than $\pm 1$ is loxodromic and has the complex length of some $\gamma\in\psi(\Ocal^1(\afrak))$, so $\tr\lambda = \pm\tr\gamma$ by \eqref{eq:tracelength}. This is condition (i) of Proposition \ref{prop:key}. In particular $\tr\tilde\Lambda\subseteq R$.

We check that $\Lambda$ is arithmetic using Theorem \ref{thm:takeuchi}(2). The traces of $\tilde\Lambda$ are algebraic integers. The field $k\Lambda\subseteq k$ is not contained in $\mathbf{R}$. Indeed, $\Lambda^{(2)}$ has finite index in $\Lambda$, so it is a lattice in $\PSL(2,\mathbf{C})$. If the traces of $\tilde\Lambda^{(2)}$ were real, then $\Lambda^{(2)}$, being non-elementary and discrete, would be conjugate into $\PSL(2,\mathbf{R})$ \cite[Chapter 3]{Maclachlan.Reid-[ArithmeticHyperbolic3Manifolds]2003}, and it would preserve a totally geodesic plane in $\mathbf{H}^3$, which is incompatible with finite covolume. If $k\Lambda$ had two complex places, they would extend to two distinct complex places of $k$, which is impossible. So $k\Lambda$ has exactly one complex place, namely the restriction of $\nu_0$. Let $\sigma$ be a real embedding of $k\Lambda$ and $\tau$ an embedding of $k$ extending $\sigma$. Then $\tau$ is real, because the complex embeddings of $k$ restrict to nonreal embeddings of $k\Lambda$. Since $B$ is ramified at $\tau$, the group $\tau(\Ocal^1)$ lies in the compact group $\mathcal{H}^1$, so $|\tau(\trd y)|\leq 2$ for $y\in\Ocal^1$. If $t\in\tr\tilde\Lambda^{(2)}\subseteq k\Lambda$, then by (i) we have $t = \pm\trd y$ for some $y\in\Ocal^1$, so $|\sigma(t)| = |\tau(\trd y)| \leq 2$. Thus $\sigma(\tr\tilde\Lambda^{(2)})$ is bounded for every real embedding $\sigma$ of $k\Lambda$, and $\Lambda$ is arithmetic.

By Theorem \ref{thm:chlr}, $M$ is commensurable with $X$, which is condition (ii) of Proposition \ref{prop:key}. Condition (iii) holds since $\vol(M) = \vol(X)$. Proposition \ref{prop:key} shows that $M$ is isometric to $X_{\Ocal'}(\afrak)$ for some maximal order $\Ocal'$. For the last assertion, let $M$ be a closed orientable hyperbolic $3$--manifold with the same complex length spectrum as $X$. By \cite{Salvai-[LaplaceComplexLength]2002,Kelmer-[RefinementStrongMultiplicity]2014}, $M$ is isospectral to $X$, and $\CL(M) = \CL(X)$ by assumption. \qed

\section{Applications to spectral rigidity}\label{sec:applications}

\subsection{Shimura curves over \texorpdfstring{$\mathbf{Q}$}{Q}}

\begin{proof}[Proof of Corollary \ref{cor:shimura}]
An indefinite quaternion division algebra over $\mathbf{Q}$ is of Fuchsian type. Since $\mathbf{Q}$ has narrow class number one, $t(B) = 1$ by Corollary \ref{cor:typenumberone}. Since $2$ is unramified in $\mathbf{Q}$, hypothesis (H2) is the assumption $\gcd(N,D) = 1$. Any prime factor $\ell$ of $N$ satisfies $e = 1 < c(\ell)$, so (H3) holds by Lemma \ref{lemma:torsion-free}. Corollary \ref{cor:C} applies.
\end{proof}

\subsection{Hurwitz surfaces}

Let $\Delta(2,3,7)\leq\PSL(2,\mathbf{R})$ be the orientation preserving subgroup of the group generated by reflections in the sides of a hyperbolic triangle with angles $\pi/2$, $\pi/3$ and $\pi/7$, so that
\begin{align*}
\Delta(2,3,7) \cong \langle a,b,c \mid a^2 = b^3 = c^7 = abc = 1\rangle.
\end{align*}
The orbifold $\Delta(2,3,7)\backslash\mathbf{H}^2$ has area $\pi/21$, the smallest area of any hyperbolic $2$--orbifold. A closed hyperbolic surface of genus $g\geq 2$ (that is, a closed Riemann surface with its hyperbolic metric) is a Hurwitz surface if and only if it is isometric to $\Gamma\backslash\mathbf{H}^2$ for a torsion free normal subgroup $\Gamma$ of $\Delta(2,3,7)$, in which case $\Delta(2,3,7)/\Gamma$ is its group of orientation preserving isometries, of order $84(g-1)$.

The group $\Delta(2,3,7)$ is arithmetic \cite{Takeuchi-[CharacterizationArithmeticFuchsian]1975}, and an explicit description is due to Elkies \cite{Elkies-[ShimuraCurveComputations]2000}, see also \cite{Katz.Schaps.Vishne-[HurwitzQuaternionOrder]2011}. Let $K = \mathbf{Q}(\eta)$ with $\eta = \zeta_7+\zeta_7^{-1}$, the maximal totally real subfield of $\mathbf{Q}(\zeta_7)$. Then $\eta^3+\eta^2-2\eta-1 = 0$, the field $K$ has discriminant $49$ and class number one, and its ring of integers is $\mathbf{Z}[\eta]$. The three real places of $K$ send $\eta$ to $2\cos(2\pi/7)>0$ and to $2\cos(4\pi/7), 2\cos(6\pi/7)<0$. We use the first to embed $K$ in $\mathbf{R}$. The units $-1$, $\eta$ and $1+\eta$ have sign vectors $(-,-,-)$, $(+,-,-)$ and $(+,+,-)$ at these places, so the units of $K$ realize every sign pattern. Let
\begin{align*}
B = \left(\frac{\eta,\eta}{K}\right) = K\langle i,j \mid i^2 = j^2 = \eta,\ ji = -ij\rangle.
\end{align*}
Then $B$ is ramified exactly at the two real places at which $\eta<0$ \cite{Elkies-[ShimuraCurveComputations]2000}, so $B$ is of Fuchsian type with $\Ram_f(B) = \emptyset$. By Corollary \ref{cor:typenumberone}, $t(B) = 1$. Put $j' = \frac{1}{2}\left(1+\eta i+(\eta^2+\eta+1)j\right)$. Then $\Ocal = \mathbf{Z}[\eta][i,j,j']$ is a maximal order of $B$, and the elements
\begin{align*}
g_2 = \frac{ij}{\eta}, \qquad g_3 = \frac{1}{2}\left(1+(\eta^2-2)j+(3-\eta^2)ij\right), \qquad g_7 = \frac{1}{2}\left(\eta^2+\eta-1+(2-\eta^2)i+(\eta^2+\eta-2)ij\right)
\end{align*}
of $\Ocal^1$ satisfy $g_2^2 = g_3^3 = g_7^7 = -1$ and $g_7g_3g_2 = -1$. The images of $g_2$, $g_3$ and $g_7$ in $\PSL(2,\mathbf{R})$ are elliptic of orders $2$, $3$ and $7$ with product $1$. They generate a copy of $\Delta(2,3,7)$ inside the discrete group $\Gamma_\Ocal$ \cite{Elkies-[ShimuraCurveComputations]2000}. Since $\Delta(2,3,7)$ has the smallest covolume among all Fuchsian groups, $\Gamma_\Ocal = \Delta(2,3,7)$.

\begin{proof}[Proof of Corollary \ref{cor:hurwitz}]
Since $B$ is unramified at every finite prime and $2$ is inert in $K$, hypotheses (H1) and (H2) hold for every nonzero ideal $\afrak$. The only prime ramified in $K$ is $7 = \pfrak_7^3$, with $e = 3 < 6$. Hence every prime $\pfrak$ of $K$ satisfies $e(\pfrak\mid\ell)<c(\ell)$, and (H3) holds for every $\afrak \neq R$ by Lemma \ref{lemma:torsion-free}. Thus $X(\afrak) = X_\Ocal(\afrak)$ is a closed surface which normally covers the orbifold $\Delta(2,3,7)\backslash\mathbf{H}^2$, so it is a Hurwitz surface. As $t(B) = 1$, it is spectrally rigid by Corollary \ref{cor:C}.

The prime $\pfrak_7 = (2-\eta)$ has norm $7$. By strong approximation \cite[Theorem 7.7.5]{Maclachlan.Reid-[ArithmeticHyperbolic3Manifolds]2003}, we have $\Gamma_\Ocal/\Gamma_\Ocal(\pfrak_7)\cong\PSL(2,\mathbf{F}_7)$, which has order $168 = 84\cdot 2$. So $X(\pfrak_7)$ is the Hurwitz surface of genus $3$, which is Klein's quartic. The prime $13\equiv -1\bmod 7$ splits completely in $K$ into three primes of norm $13$. The corresponding surfaces have isometry group $\PSL(2,\mathbf{F}_{13})$ of order $1092 = 84\cdot 13$, so genus $14$. They form the first Hurwitz triplet. Finally, $(2)$ is a prime of norm $8$, and $X((2))$ has isometry group $\PSL(2,\mathbf{F}_8)$ of order $504 = 84\cdot 6$, so it is the Hurwitz surface of genus $7$, the Fricke--Macbeath surface.
\end{proof}

\begin{remark}
Hurwitz surfaces that are not principal congruence covers, which exist in abundance, are not covered by our methods, so Reid's question remains open in general.
\end{remark}

\subsection{The minimal volume hyperbolic \texorpdfstring{$3$}{3}--orbifold}

The three dimensional analogues of Hurwitz surfaces are the regular manifold covers of the hyperbolic $3$--orbifold of smallest volume. This orbifold was identified by Gehring--Martin \cite{Gehring.Martin-[MinimalCoVolumeHyperbolic]2009}, Marshall--Martin \cite{Marshall.Martin-[MinimalCoVolumeHyperbolic]2012}, and it is arithmetic \cite{Chinburg.Friedman-[SmallestArithmeticHyperbolic]1986}. It is $\Gamma_{\min}\backslash\mathbf{H}^3$, where $\Gamma_{\min}$ is the $\mathbf{Z}/2$ extension of the orientation preserving subgroup
\begin{align*}
 T = \left\langle x, y, z \mid x^2=y^2=z^3=(y z)^2=(z x)^5=(x y)^3=1\right\rangle
\end{align*}
of the Coxeter group $[3,5,3]$ obtained by adjoining the rotational symmetry $w$ of the tetrahedron. Setting $a = wy$ and $b = z$ yields the presentation
\begin{align*}
\Gamma_{\min} = \left\langle a, b \mid a^2=b^3=1,\ c=(a b)^2\left(a b^{-1}\right)^2,\ c^5=1,\ \left(b^{-1} c^2\right)^2=1\right\rangle .
\end{align*}

The invariant trace field of $\Gamma_{\min}$ is $k = \mathbf{Q}(t)$ with $t^4-2t^3+t-1 = 0$. This is the unique quartic field with exactly one complex place and discriminant $-275 = -5^2\cdot 11$, and it has class number one. The polynomial has two real roots $t\approx 1.8668$ and $t\approx -0.8668$, so the unit $t$ has opposite signs at the two real places, and the units of $k$ realize every sign pattern. The invariant quaternion algebra $A$ is the quaternion algebra over $k$ ramified exactly at its two real places \cite{Chinburg.Friedman-[SmallestArithmeticHyperbolic]1986}. It is of Kleinian type with $\Ram_f(A) = \emptyset$, and $t(A) = 1$ by Corollary \ref{cor:typenumberone}. If $\Ocal$ is a maximal order of $A$, then $\Gamma_\Ocal\cong T$ has index two in $\Gamma_{\min}$, and $\Gamma_{\min}$ is the image of $\Ocal^\times$ in $\PGL(2,\mathbf{C})$ \cite{Chinburg.Friedman-[SmallestArithmeticHyperbolic]1986}. Since $\Ocal^\times$ normalizes $\Ocal$, each $\Gamma_\Ocal(\afrak)$ is normal in $\Gamma_{\min}$.

\begin{corollary}\label{cor:min-orbifold}
Let $\afrak\neq R$ be a nonzero ideal of the ring of integers $R$ of $k$. Then $X_\Ocal(\afrak)$ is a closed hyperbolic $3$--manifold that regularly covers $\Gamma_{\min}\backslash\mathbf{H}^3$. Every closed orientable Riemannian manifold $M$ that is isospectral to $X_\Ocal(\afrak)$ and satisfies $\CL(M)\subseteq\CL(X_\Ocal(\afrak))$ is isometric to $X_\Ocal(\afrak)$.
\end{corollary}

\begin{proof}
The primes ramified in $k$ lie over $5$ and $11$. Wild ramification over a prime $\ell$ would force $v_\ell(-275)\geq\ell$, so the ramification is tame and $v_\ell(-275) = \sum_{\pfrak\mid\ell}(e(\pfrak\mid\ell)-1)f(\pfrak\mid\ell)$. As $v_5(-275) = 2$ and $v_{11}(-275) = 1$, every prime over $5$ has ramification index at most $3<4$, and every prime over $11$ has ramification index at most $2<10$. All other primes are unramified, with $e = 1 < c(\ell)$. In particular $2$ is unramified in $k$, so (H2) holds for every $\afrak$ because $\Ram_f(A) = \emptyset$, and (H3) holds by Lemma \ref{lemma:torsion-free}. Hypothesis (H1) was verified above. By Theorem \ref{thm:B}, $M$ is isometric to $X_{\Ocal'}(\afrak)$ for some maximal order $\Ocal'$, and since $t(A) = 1$, this manifold is isometric to $X_\Ocal(\afrak)$.
\end{proof}


\medskip
\noindent
Justin Katz: Ohio State University, \texttt{jukatz93@gmail.com}

\medskip
\noindent
D. B. McReynolds: Purdue University, \texttt{dmcreyno@purdue.edu}

\end{document}